\documentclass{amsart}
\usepackage{amsmath, amsthm, amssymb, mathtools}
\usepackage{graphicx}
\graphicspath{{images/}}
\newtheorem{theorem}{Theorem}[section] 
\newcommand{\bt}{\begin{theorem}} 
\newcommand{\et}{\end{theorem}} 

\newtheorem{corollary}[theorem]{Corollary}   
\newcommand{\bc}{\begin{corollary}} 
\newcommand{\ec}{\end{corollary}} 

\newtheorem{ex}[theorem]{Example}  

\newtheorem{lemma}[theorem]{Lemma}  
\newcommand{\bl}{\begin{lemma}} 
\newcommand{\el}{\end{lemma}} 

\newtheorem{proposition}[theorem]{Proposition} 
\newcommand{\bp}{\begin{proposition}} 
\newcommand{\ep}{\end{proposition}} 

\newtheorem{definition}[theorem]{Definition}
\newcommand{\bd}{\begin{definition}}  
\newcommand{\ed}{\end{definition}} 

\newtheorem{remark}[theorem]{Remark} 
\newcommand{\br}{\begin{remark}} 
\newcommand{\er}{\end{remark}} 

\newcommand{\R}{\mathbb{R}}
\newcommand{\Sp}{\mathbb{S}}
\newcommand{\Fto}{\stackrel {\mathcal{F}}{\longrightarrow} }

\DeclareMathOperator{\len}{L}
\DeclareMathOperator{\Vol}{Vol}
\DeclareMathOperator{\Diam}{Diam}
\DeclareMathOperator{\Area}{Area}

\newcommand{\lip}{\operatorname{Lip}}
\newcommand{\mass}{{\mathbf M}}
\def\set{\textrm{set}}
\newcommand{\intcurr}{{\mathbf I}} 
\newcommand{\VFto}{\stackrel {\mathcal{VF}}{\longrightarrow} }

\begin{document}
\title{Volume Above Distance Below with Boundary II}

\author{Brian Allen}
\address{Lehman College, CUNY}
\email{brianallenmath@gmail.com}

\author{Edward Bryden$^\dagger$}
\address{Universiteit Antwerpen}
\email{etbryden@gmail.com}
\thanks{$^\dagger$funded by FWO grant 12F0223N}

\begin{abstract}
 It was shown by B. Allen, R. Perales, and C. Sormani \cite{Allen-Perales-Sormani-VADB} that on a closed manifold where the diameter of a sequence of Riemannian metrics is bounded, if the volume converges to the volume of a limit manifold, and the sequence of Riemannian metrics are $C^0$ converging from below then one can conclude
  volume preserving Sormani-Wenger Intrinsic Flat convergence. The result was extended to manifolds with boundary by B. Allen and R. Perales \cite{Allen-Perales} by a doubling with necks procedure which produced a closed manifold and reduced the case with boundary to the case without boundary. The consequence of the doubling with necks procedure was requiring a stronger condition than necessary on the boundary. Using the estimates for the Sormani-Wenger Intrinsic Flat distance on manifolds with boundary developed by B. Allen and R. Perales \cite{Allen-Perales}, we show that only a bound on the area of the boundary is needed in order to conclude volume preserving
  intrinsic flat convergence for manifolds with boundary. We also provide an example which shows that one should not expect convergence without a bound on area.
\end{abstract}
\maketitle
\section{Introduction}
An important notion of convergence one can use for
studying scalar curvature stability problems is the Sormani-Wenger Intrinsic
Flat (\textit{SWIF}) convergence \cite{SW-JDG}. Since the definition of SWIF convergence is rooted in the
theory of currents on general metric spaces,
it is helpful to have a theorem which identifies geometric assumptions
on a sequence of Riemannian manifolds which are sufficient to conclude SWIF convergence, and
so simplify the application of SWIF convergence to Riemannian manifolds.
One such theorem, called the VADB Theorem, was provided by B. Allen, R. Perales, and C. Sormani
\cite{Allen-Perales-Sormani-VADB}. It states that on a closed manifold, if the diameter of a sequence of Riemannian
metrics is bounded, the volume converges to the volume of a limit manifold, and the Riemannian metric is
$C^0$ converging from below then one can conclude volume preserving SWIF convergence.
This was extended by B. Allen and R. Perales \cite{Allen-Perales} to manifolds with boundary by a doubling with necks procedure which produced a closed manifold and reduced the case with boundary to the case without boundary. The price of the doubling with necks procedure was an additional hypothesis on the boundary which was sufficient but not necessary. In this note we remove this additional hypothesis
on the boundary and are able to prove the desired theorem with just a bound on the area of the
boundary. 

An important scalar curvature stability result which used the VADB theorem was the resolution
of the stability of Llarull's theorem on the sphere.
This was accomplished by B. Allen, E. Bryden,
and D. Kazaras \cite{ABKLLarull} in dimesion $n=3$ using spacetime harmonic functions and
extended by S. Hirsch and Y. Zhang \cite{HZ} for $n \ge 3$ using spinors.
Both proofs heavily rely on the VADB theorem to conclude SWIF convergence.
In addition, in the same paper \cite{HZ}, S. Hirsch and Y. Zhang are able to establish the stability
of Llarull's theorem in the sense of C. Dong and A. Song \cite{Dong-Song} for Lipschitz 1 maps from arbitrary spin manifolds
into the sphere.
Interestingly, the convergence defined by C. Dong and A. Song  doesn't arise from a metric,
or even a topology, on the space
of Riemannian manifolds, unlike Sormani-Wenger Intrinsic Flat convergence or $d_p$ convergence first defined by  M-C. Lee, A. Naber, and R. Neumayer \cite{LNN}. See C. Sormani \cite{SormaniIAS} and B. Allen \cite{AllenOberwolfach} for more discussion of scalar curvature stability problems. Other applications of the VADB theorem to scalar curvature stability problems include conformal
cases of Geroch stability by  B. Allen \cite{Allen-Conformal-Torus}, J. Chu and M.-C. Lee 
\cite{JCMCL}, the graph case of Geroch stability by  A. Pacheco,  C. Ketterer and
R. Perales \cite{PKP19}, the graph case of stability of the positive mass theorem by C. Sormani,
L.-H. Huang, and D. Lee \cite{HLS}, B. Allen and R. Perales \cite{Allen-Perales}, R. Perales,
L.-H. Huang and D. Lee \cite{Huang-Lee-Perales} in the asymptotically Euclidean case, and
A. Pacheco,  M. Graf and R. Perales \cite{PGPAsymHyp} in the asymptotically hyperbolic case.

Since there are scalar curvature stability problems where it is natural to consider manifolds with boundary, it is important that a VADB theorem with minimal assumptions on the boundary is available in this case. 
We see by Example \ref{ex-Area Removed} that a bound on area is the minimal boundary condition one can assume to show SWIF convergence. Additionally, by Example \ref{ex-Cinched-Sphere} and Example \ref{ex-NoL^mConv}, originally appearing in the work of B. Allen and C. Sormani \cite{Allen-Sormani-2}, it is known that one should not expect SWIF convergence in the
absence of $C^0$ convergence from below of the Riemannian metrics or volume convergence.
The VADB theorem proven below shows that these minimal hypotheses imply SWIF convergence.

\begin{theorem}\label{thm:Main Theorem}
Let $M^n$ be a compact, connected, and oriented manifold with boundary, $g_j$ a sequence of continuous Riemannian metrics, and $g_0$ a smooth Riemannian metric. If 
\begin{align}
\Diam(M,g_j) &\le D,
  \\  \Vol(M,g_j) &\rightarrow \Vol(M,g_0)
  \\ \Area(\partial M,g_j) &\le A,
  \\ g_j(v,v) &\ge ( 1 - C_j)g_0(v,v), \quad \forall p \in M, v \in T_pM, \quad C_j \searrow 0,
\end{align}
then $(M,g_j)$ converges in the volume preserving Sormani-Wenger Intrinsic Flat sense to $(M,g_0)$.
\end{theorem}

One can also think of Theorem \ref{thm:Main Theorem} as a stability theorem for the following rigidity result. If $(M,g_1)$ and $(M,g_2)$ are two Riemannian manifolds so that $g_1 \ge g_2$ and $\Vol(M,g_1) =\Vol(M,g_2)$ then $g_1=g_2$. This way of viewing the VADB theorem was first pointed out by B. Allen, R. Perales, and C. Sormani (see Theorem 2.3 of \cite{Allen-Perales-Sormani-VADB}). Theorem \ref{thm:Main Theorem} extends this stability result to the case of Reimannian manifolds with boundary so that the assumption on the boundary is minimal.

In section \ref{sec-background}, we review the definition of SWIF distance between integral current spaces. In addition, we review the construction of special metric spaces $Z$ which were originally constructed in \cite{Allen-Perales-Sormani-VADB} and was extended to the case of manifolds with boundary in \cite{Allen-Perales}. Several other foundational theorems from \cite{Allen-Perales} are also reviewed which will be used to prove Theorem \ref{thm:Main Theorem}.

In Section \ref{sec-Examples}, we review examples given in \cite{Allen-Perales-Sormani-VADB} which show the necessity of the assumptions made in Theorem \ref{thm:Main Theorem}. We also give a new example sequence with boundary which shows that the assumption of bounded area of the boundary is the minimal assumption one could make on the boundary.

In Section \ref{sec-Pointwise Convergence}, we proved that under the hypotheses of Theorem \ref{thm:Main Theorem} we can conclude pointwise almost everywhere convergence. This is accomplished by doing a conformal change near the boundary to make the new manifold convex, then applying a theorem from \cite{Allen-Perales}, and showing that the result about the conformally changed sequence gives us the desired conclusion for the original sequence. We expect that this method could be useful for reducing the convergence theorems for sequences of Riemannian manifolds with boundary to the convex case in general. The proof of Theorem \ref{thm:Main Theorem} is then finished by applying a result of B. Allen and R. Perales \cite{Allen-Perales} which allows one to go from pointwise almost everywhere convergence to SWIF convergence.

\section{Background}\label{sec-background}

In this section we review definitions and theorems from previous papers which are essential to understanding the proof of Theorem \ref{thm:Main Theorem}.

\subsection{Sormani-Wenger Intrinsic Flat Distance Review}

We start by reviewing the definition of the flat distance of H. Federer and W. H. Fleming \cite{FF} which was extended to arbitrary metric spaces by L. Ambrosio and B. Kircheim \cite{AK}.

Let $(Z,d)$ be a complete metric space, $\lip(Z)$ the set of  real valued Lipschitz functions on $Z$, and 
$\lip_b(Z)$ the bounded ones.  An $m$-dimensional current $T$ on $Z$ 
is a multilinear map $T: \lip_b(Z) \times [\lip(Z)]^n  \to \R$  that satisfies properties which can be found in Definition 3.1 of  \cite{AK}.  
From the definition of $T$ we know there exists a finite Borel measure on $Z$,  $||T||$,  called the mass measure of $T$.
Then the mass of $T$ is defined as $\mass(T)=||T||(Z)$.  The  boundary of $T$,  $\partial T: \lip_b(Z) \times [\lip(Z)]^{n-1}  \to \R$  
is the linear functional given by
\begin{align}
\partial T(f, \pi) = T(1, (f, \pi)),
\end{align}
and for any Lipschitz function $\varphi: Z \to Y$  the push forward of $T$,   ${\varphi}_{\sharp} T : \lip_b(Y) \times [\lip(Y)]^{m}  \to \R$
is the $n$-dimensional current given by 
\begin{align}
{\varphi}_{\sharp} T (f, \pi) 
= T( f\circ \varphi, \pi \circ \varphi ).
\end{align}
Furthermore,  the following inequality holds
\begin{equation}\label{eq-pushMeasure}
||  \varphi_\sharp T||  \leq \lip(\varphi)^n \varphi_\sharp ||T||. 
\end{equation} 
 
An $n$-dimensional integral current in $Z$ is an $n$-dimensional current that can be written as a countable sum of terms, 
\begin{align}
    T=  \sum_{i=1}^\infty \varphi_{i\sharp} [[\theta_i]],
\end{align} 
with $\theta_i \in L^1(A_i, \mathbb R)$ integer constant functions, 
such that $\partial T$ is a current.  The class that contains all $n$-dimensional integral currents of $Z$ is denoted by $\intcurr_n(Z)$. 
For  $T\in I_n(Z)$,  L. Ambrosio and B. Kirchheim proved that the subset 
\begin{align}
\set(T)= \left\{ z \in Z \, | \, \liminf_{r \downarrow 0} \frac{\|T\|(B_r(z))}{ r^n }> 0 \right\}
\end{align}   
is $\mathcal H^n$-countably recitifiable. That is, $\set(T)$ can be covered by images of Lipschitz maps from $\R^n$  to $Z$
up to a set of zero  $\mathcal H^n$-measure.

The flat distance between two integral currents $T_1, T_2  \in  \intcurr_{n} (Z)$ is defined as
\begin{align}
\begin{split}
d_{F}^Z( T_1, T_2)=\inf\Bigl\{  \mass(U)+ \mass(V)\,  |&
\, \, U \in \intcurr_{n}(Z), \, V   \in  \intcurr_{n+1} (Z),
\\
 & \, \, T_2 -T_1 =U + \partial V  \Bigr\}.    
\end{split}
\end{align}

With the definition of flat convergence on a general metric space in hand we are ready to define integral current spaces which are the spaces for which Sormani-Wenger intrinsic flat distance is defined. One should see C. Sormani and S. Wenger \cite{SW-JDG} for more details. An $m$-dimensional integral current space $(X, d, T)$ consists of a metric space $(X, d)$ and an $m$-dimensional integral current defined on the completion of $X$, $T\in I_n(\bar{X})$, such that $\set(T)=X$.  Since in this paper we will be mostly considering the SWIF distance between Riemannian manifolds it is helpful to see how the integral current space structure applies in this case.

\begin{ex}
For an  $n$-dimensional compact oriented $C^0$ Riemannian manifold $(M^n,g)$,
the triple $(M,d_g, [[M]])$ is an $n$-dimensional integral current space.
Here
$d_g$ is the usual metric induced by $g$. Then $[[M]] :  \lip_b(M) \times [\lip(M)]^n \to \R$ is given by 
\begin{align}\label{eq-canonicalT}
[[M]] = &  \sum_{i,k} {\psi_i}_\sharp [[1_{A_{ik}}]] 
\end{align}
where $\Bigl\|[[M]]\Bigr\|= dVol_g$ and we have chosen $\{(U_i, \psi_i)\}_{i \in \mathbb N}$ a $C^1$ locally finite atlas  of $M$ consisting of positively oriented Lipschitz charts, 
$\psi_i :  U_i  \subset \R^n   \to M$  and $A_{ik}$  are precompact Borel sets such that $\psi_i(A_{ik})$ have disjoint images for all $i$ and $k$ and cover $M$ for $\mathcal H^n$-almost everywhere.
\end{ex}

We say that an integral current space $(X,d,T)$ is precompact if $X$ is precompact with respect to $d$.
Given two $m$-dimensional integral current spaces, $(X_1, d_1, T_1)$ and $(X_2, d_2, T_2)$, a current preserving isometry between them is a 
 metric isometry $\varphi: X_1 \to X_2$  such that $\varphi_\sharp T_1=T_2$.   We are now ready to state the definition of the SWIF distance between integral current spaces.

\begin{definition}[Sormani-Wenger \cite{SW-JDG}]\label{defn-SWIF} 
Given two $m$-dimensional precompact integral current spaces 
$(X_1, d_1, T_1)$ and $(X_2, d_2,T_2)$, the Sormani-Wenger Intrinsic Flat distance   
between them is defined as
\begin{align}
\begin{split}
d_{\mathcal{F}}&\left( (X_1, d_1, T_1), (X_2, d_2, T_2)\right)
\\&=\inf  \Bigl\{d_F^Z(\varphi_{1\sharp}T_1, \varphi_{2\sharp}T_2)|  \,
 (Z,d_Z) \text{ complete} ,\, \varphi_j: X_j \to Z   \text{ isometries}\Bigr\}.         
\end{split}
\end{align}
\end{definition}

The function $d_{\mathcal{F}}$ is a distance up to current preserving isometries. One should note that if a sequence  $(X_j,d_j,T_j)$ of $n$-dimensional integral current spaces 
converges in the SWIF sense to $(X,d,T)$ then it follows that
\begin{align}\label{Eq-Boundary Convergence}
    (\partial X_j, d_j, \partial T_j) \Fto (\partial X, d, \partial T),
\end{align}
in the SWIF sense as well. We say that a sequence $(X_j,d_j,T_j)$ of $n$-dimensional integral current spaces converges in the volume preserving Sormani-Wenger Intrinsic Flat sense, $\mathcal{VF}$, to $(X,d,T)$ if the sequence converges with respect to the intrinsic flat distance
to $(X,d,T)$ and the masses $\mass(T_j)$ converge to $\mass(T)$.  

\subsection{Volume Above Distance Below with Boundary I Review}

In order to estimate the Sormani-Wenger Intrinsic Flat distance, B. Allen and R. Perales \cite{Allen-Perales} made the following definition of a common metric space $Z$, building off of the definition given by Allen-Perales-Sormani \cite{Allen-Perales-Sormani-VADB}.

\begin{definition}[Definition 4.5 in Allen and Perales \cite{Allen-Perales} ]\label{defn-Z}
Let $M$ be a compact manifold, $M_j=(M,g_j)$ and $M_0=(M,g_0)$ be continuous Riemannian manifolds, $F_j: M_j \rightarrow M_0$ a bijective map and $W_j \subset M_j$. Define the set
\begin{align}
Z : =  M_0  \sqcup \left(   M \times [0,h_j] \right) \sqcup  M_j  \,\,|_\sim
\end{align}
where  $x \sim (F_j^{-1}(x),0)$ for all $x \in M_0$ and  $x \sim (x,h_j)$ for all $x \in W_j$.
Define the function $d_Z: Z \times Z \to [0, \infty)$  by
\begin{align}
d_Z(z_1, z_2) = \inf \{L_Z(\gamma):\, \gamma(0)=z_1,\, \gamma(1)=z_2\}
\end{align}
where $\gamma$ is any piecewise smooth curve joining $z_1$ to $z_2$ and the length function 
$L_Z$ is defined as follows,  $L_Z|_{M_j} = L_{g_j}$,  $L_Z|_{M_0} = L_{g_0}$ and $L_Z|_{M \times (0,h_j]} = L_{g_j +  dh^2}$, the length space metric associated to $g_j+dh^2$.

Define functions $\varphi_0: M_0 \to Z$ and $\varphi_j: M_j \to Z$ by 
\begin{align}
\varphi_0(x) = & (F_j^{-1}(x), 0) \\
\varphi_j(x) =   &
\begin{cases}
x   & x \notin \overline{W}_j \\
(x, h_j)  &  \textrm{otherwise.} 
\end{cases}
\end{align}
\end{definition}

Now B. Allen and R. Perales use the metric space defined in Definition \ref{defn-Z} to estimate the Sormani-Wenger Intrinsic Flat distance \ref{defn-SWIF} by calculating the Flat distance between two Riemmanian manifolds isometrically embedded in $Z$.

\begin{theorem}[Theorem 4.7 in Allen-Perales \cite{Allen-Perales}]\label{thm-est-SWIF}
Let $M$ be an oriented and compact manifold, $M_j=(M,g_j)$ and $M_0=(M,g_0)$ be continuous Riemannian manifolds with  
\begin{align}
  \Diam(M_j) \le D, \qquad  \Vol(M_j)\le V,\qquad \Vol(\partial M_j) \leq A, 
\end{align}
and $F_j: M_j \rightarrow M_0$ a biLipschitz  and distance non-increasing map with a $C^1$ inverse. 
Let $W_j \subset  M_j$ be a measurable set with 
\begin{align}\label{eq-volCond}
\Vol( M_j \setminus W_j) \le V_j
\end{align}
and assume that there exists a $\delta_j > 0$ so that for all $ x,y \in W_j$,
 \begin{align}\label{eq-distCond}
d_j(x,y) \le d_0( F_j(x),  F_j(y)) +2 \delta_j
\end{align}
and that $h_j = \sqrt{2 \delta_j D + \delta_j^2}$.
Then
\begin{align}\label{Fest}
d^Z_{F}( \varphi_0(M_0), \varphi_j(M_j)) \le 2V_j + h_j V +  h_j A 
\end{align}
where  $Z$ is the space described in Definition \ref{defn-Z}.
\end{theorem}

The method for applying Theorem \ref{thm-est-SWIF} is to show pointwise a.e. convergence of the sequence of distance functions and apply Egeroff's Theorem in order to conclude uniform convergence on a good set $W_j$. By implementing this method B. Allen and R. Perales were able to prove the following theorem which is the main tool we will combine with the results of the next section in order to prove Theorem \ref{thm:Main Theorem}.

\begin{theorem}[Theorem 4.1 in Allen-Perales \cite{Allen-Perales}]\label{thm-IF conv from Pointwise ae}
Let $M$ be a compact oriented manifold, $g_0$  a smooth Riemannian metric, and $g_j$ a sequence of continuous Riemannian metrics
such that
\begin{align} 
g_0(v,v) \le g_j(v,v), \quad \forall v \in T_pM,
\end{align}
\begin{align}
\Diam(M,g_j) \le D,
\end{align}
\begin{align}
\Vol(M,g_j) \to  \Vol(M,g_0),
\end{align}
\begin{align}
\Area(\partial M,g_j) \leq A,
\end{align}
and
\begin{align}
d_j(p,q)\to d_0(p,q),\textrm{ for } dV_{g_0} \times dV_{g_0} \textrm{ a.e. } (p,q)
\in M \times M. 
\end{align}
Then
\begin{align}
(M,g_j) \VFto (M,g_0).
\end{align}
\end{theorem}

B. Allen and R. Perales also observed that when the boundary of a Riemannian manifold is convex
the boundary case of Theorem \ref{thm:Main Theorem} is no different to prove than the compact
without boundary case proved by B. Allen, R. Perales, and C. Sormani \cite{Allen-Perales-Sormani-VADB}.
We now state the thoerem which follows from this observation which will be used in the proof of
Theorem \ref{thm:a.e_convergence_on_distance_level_sets}.

\begin{theorem}[Theorem 4.2 in Allen-Perales \cite{Allen-Perales}]\label{thm-Pointwise ae Convergence Convex Boundary}
Let $M$ be a compact oriented manifold, $g_0$ a smooth Riemannian metric, and $g_j$ a sequence of  continuous Riemannian metrics
such that
\begin{align}
g_0(v,v) \le g_j(v,v), \quad \forall v \in T_pM,
\end{align}
\begin{align}
\Diam(M,g_j) \le D,
\end{align}
\begin{align}\label{volsconv}
\Vol(M,g_j) \to  \Vol(M,g_0),
\end{align}
\begin{align}
\Area(\partial M,g_j) \leq A,
\end{align}
and  the interior of $(M,g_0)$ is convex, i.e. for all $p,q$ in the interior of $M$, 
and $\gamma:[0,1]  \to M$ $g_0$ minimizing geodesic joining $p$ to $q$ 
we have that $\gamma(I)$ remains in the interior of $M$.  Then
\begin{align}\label{convae2}
d_j(p,q)\to d_0(p,q),  \qquad \textrm{ for } dV_{g_0} \times dV_{g_0} \textrm{ a.e. } (p,q) \in M \times M. 
\end{align}
\end{theorem}

\section{Examples}\label{sec-Examples}

In this section we give three examples which show the necessity of the assumptions made in Theorem \ref{thm:Main Theorem}.

\subsection{Area Bound Removed}
Here we see an example where if we remove the bound on area in Theorem \ref{thm:Main Theorem} we cannot conclude SWIF convergence.

\begin{ex}\label{ex-Area Removed}
    Let $g_0$ be the standard flat metric on the unit disk $\mathbb{D}^n$. Let $g_j=f_j^2g_0$ be metrics conformal to $g_0$ with smooth conformal factors, $f_j$ that are radially defined from the boundary of the disk as follow:
    \begin{equation}
f_j(r)=
\begin{cases}
j^{\alpha} &\text{ if } r \in [0,1/j]
\\h_j(jr) &\text{ if } r \in [1/j,2/j]
\\ 1 & \text{ if } r \in (2/j,\sqrt{m}\pi].
\end{cases}
\end{equation}
where $0<\alpha < \frac{1}{n}$. Then $  (\mathbb{D}^n, g_j)$ satisfies 
\begin{align}
g_j &\ge g_0, \quad \Vol(\mathbb{D}^n,g_j) \to \Vol(\mathbb{D}^n,g_0),  \quad 
\\\Diam(\mathbb{D}^n,g_j) &\le D_0, \quad \Area(\partial \mathbb{D}^n,g_j) \rightarrow \infty,
\end{align}
and we find that $(\mathbb{D}^n,g_j)$ does not converge in the SWIF sense.
\end{ex}
\begin{remark}
    One can show that $\mathbb{D}^n_j$ does converge in the measured Gromov-Hausdorff sense to the flat disk. That being said, by adding increasingly many splines to Example \ref{ex-Area Removed} one would obtain an example which does not converge in the Gromov-Hausdorff or SWIF senses.
\end{remark}

\begin{proof}
We let $\omega_n$ to be the volume in a Euclidean ball of radius one. We notice that by construction $f_j \ge 1$ so that $g_j \ge g_0$ and $\Vol(\mathbb{D}^n,g_j) \ge \Vol(\mathbb{D}^n,g_0)$. Now we can calculate an upper bound on volume
\begin{align}
    \Vol(\mathbb{D}^n,g_j)&\le j^{\alpha n} \omega_n\left(1-\left(1-\frac{2}{j}\right)^n \right)+\omega_n \left(1-\frac{1}{j}\right)^n
    \\&\rightarrow \omega_n = \Vol(\mathbb{D}^n,g_0).
\end{align}
Hence we see that the volume converges to that of the flat disk. The diameter is bounded since the length of any radial curve from the boundary to the center of the disk  is given by
\begin{align}
    \frac{1}{2}\int_0^1 f_j(r)dr \le \frac{j^{\alpha}}{j}+ \frac{1}{2}\left(1-\frac{1}{j}\right) \le 2.
\end{align}
Lastly we see that the area is given by
\begin{align}
   \Area(\partial \mathbb{D}^n,g_j)&= \partial j^{\alpha(n-1)} \Area(\mathbb{D}^n,g_0) \rightarrow \infty.
\end{align}
Hence we see by \eqref{Eq-Boundary Convergence} that there can be no SWIF limit.
\end{proof}

\subsection{$C^0$ Convergence From Below Removed}

Here we see an example where if we remove the $C^0$ convergence from below in Theorem \ref{thm:Main Theorem} we cannot conclude SWIF convergence to the desired Riemannian manifold.

\begin{ex}[Example 3.1 in  \cite{Allen-Sormani-2}]\label{ex-Cinched-Sphere}
Let $g_0$ be the standard round metric on the sphere ${\mathbb S}^n$.   Let $g_j=f_j^2 g_0$ be
metrics conformal to $g_0$ with smooth conformal factors, $f_j$,
that are radially defined from the north pole with a cinch at the equator as follows:  
 \begin{align}
 f_j(r)=
 \begin{cases}
 1 & r\in[0,\pi/2- 1/j]
 \\  h(j(r-\pi/2)) & r\in[\pi/2- 1/j, \pi/2+ 1/j]
 \\ 1 &r\in [\pi/2+ 1/j, \pi]
 \end{cases}
\end{align}
where $h:[-1,1]\rightarrow \R$ is an even function 
decreasing to $h(0)=h_0\in (0,1)$ and then
increasing back up to $h(1)=1$.   Then $\Sp^m_j =  (\Sp^m, g_j)$ satisfies 
\begin{align}
g_j \le g_0, \quad \Vol(\Sp_j^m) \to \Vol(\Sp_0^n,g_0),  \quad \Diam(\Sp^n,g_j) \le \Diam(\Sp^n,g_0).
\end{align}
But  $(\Sp^n,g_j) \Fto (\Sp^n,g_{\infty})$, where 
$g_\infty=f_\infty^2g_0$ with $ f_{\infty}(r)=h_0$ for  $r=\pi/2$ and  $ f_{\infty}(r)=1$ otherwise. 
\end{ex}

\subsection{Volume Convergence Removed}

Here we see an example where if we remove the volume convergence in Theorem \ref{thm:Main Theorem} we cannot conclude SWIF convergence to the desired Riemannian manifold. 

\begin{ex}[Example 3.5 in \cite{Allen-Sormani-2}]\label{ex-NoL^mConv}
Let  $(\mathbb T^n, g_0)$ be a torus and $h_j:[1,2] \rightarrow [1, \infty)$ be a smooth, decreasing function so that $h_j(1) = j$, $h_j'(1)=h_j'(2)=0$, and $h_j(2) = 1$ so that
\begin{align}
    \frac{1}{j^n}\int_1^2h_j(s)^n s^{n-1}ds \rightarrow 0.\label{ConstructionHyp}
\end{align}
Given a point $p \in \mathbb T^n$,  consider the sequence of functions $f_j:  \mathbb T^n  \to [1,\infty)$ which are radially defined from $p$ by
\begin{equation}
f_j(r)=
\begin{cases}
j &\text{ if } r \in [0,1/j]
\\h_j(jr) &\text{ if } r \in [1/j,2/j]
\\ 1 & \text{ if } r \in (2/j,\sqrt{n}\pi].
\end{cases}
\end{equation}
Then the sequence $(\mathbb T^n, f_j^2 g_0)$ satisfies 
\begin{align}
g_0 \leq g_j,  \quad \Diam(\mathbb T^m_j) \leq 1+ \sqrt{n}\pi, \\
\Vol( \mathbb T^n,g_j )  \to \Vol(B(p,1),g_0)+\Vol(\mathbb T^n, g_0). 
\end{align}
Furthermore,  it converges in intrinsic and flat sense to a torus with a bubble attached, where if $F:\partial \mathbb{D} \rightarrow \mathbb{T}^n$ is given by $F(d)=p$ for a fixed point $p \in \mathbb{T}^n$ and $d_{\infty}$ is the metric on $\mathbb T^m \sqcup \mathbb{D}^n)$ then the sequence converges to
\begin{align}
    (\mathbb T^n \sqcup\mathbb{D}^n, d_\infty/\sim), 
\end{align}
where points are identified through the map $F$.
\end{ex}

\section{Pointwise Almost Everywhere Convergence}\label{sec-Pointwise Convergence}

In this section we will show how to conclude pointwise almost everywhere convergence on a manifold with boundary. By Theorem \ref{thm-IF conv from Pointwise ae} we note that this is enough to prove Theorem \ref{thm:Main Theorem}. We start by fixing some notation for the portion of the manifold which stays a fixed distance from the boundary.

\begin{definition}
  Let $(M,\partial{}M,g_{0})$ be a smooth Riemannian manifold with boundary.
  We let $M_{t}=\{x:d_{0}(x,\partial{}M)<t\}$.
\end{definition}

Now by using the exponential map defined on the boundary of a Riemannian manifold we see that any points $p,q \in M_t^c$ can be connected by a curve in this region whose length is close to the distance between $p$ and $q$.

\begin{lemma}\label{lem:rough_approximate_geodesic}
  Let $(M,\partial{}M,g_{0})$ be a smooth Riemannian manifold with boundary.
  There is a $t_{0}>0$ and a fixed function $O(t)$ such that for any $t\leq t_{0}$
  and $p,q\in{}M_{t}^{c}$ we may find a curve $\gamma\subset{}M_{t}^{c}$
  connecting $p$ to $q$ such that
  \begin{equation}
    \len_{g_{0}}(\gamma)\leq d_{g_{0}}(p,q)+O(t).
  \end{equation}
\end{lemma}

\begin{proof}
  As $(M,\partial{}M,g_{0})$ is a smooth Riemannian manifold, there is a $t_{0}$
  such that $\exp:\partial{}M\times{}[0,t_{0})\rightarrow{}M_{t_{0}}$ is
  a diffeomorphism.
  Furthermore, we have that
  \begin{equation}
    \exp^{*}g_{0}=dt^{2}+h(t),
  \end{equation}
  where $h(t)$ is a family
  of metrics on $\partial{}M$ such that $h(t)=\left.g_{0}\right|_{\partial{}M}+O(t)$.

  Let $c(s)$ be a unit speed geodesic connecting $p$ to $q$, and observe that
  for segments of $c$ lying in $M_{t_{0}}$ we may decompose $c$ into tangential
  and normal parts: 
  \begin{equation}
    c(s)=\Bigl(c^{\mathrm{T}}(s),c^{\mathrm{N}}(s)\Bigr).
  \end{equation}
  We define $\gamma^{t}_{p,q}$ as follows:
  \begin{equation}
    \gamma^{t}_{p,q}(s)=
    \begin{cases}
      c(s)&\text{ if }c(s)\in{}M_{t}^{c}
      \\
      \Bigl(c^{\mathrm{T}}(s),t\Bigr)&\text{ if }c(s)\in{}M_{t}
    \end{cases}
    .
  \end{equation}
  We observe that for all $s$ such that $c(s)\in{}M_{t}$ we have
  \begin{align}\label{eq:bound_norm_of_tangent_part}
    \lvert \dot{\gamma}^{t}_{p,q}(s)\rvert^{2}_{h(t)}&\leq
    \lvert\dot{c}^{\mathrm{N}}(s)\rvert^{2}+ \lvert{}\dot{c}^{\mathrm{T}}(s)\rvert^{2}_{h(t)}
    \\
                                                     &=\lvert{}\dot{c}(s)\rvert^{2}_{g_{0}}
                                                     +\lvert\dot{c}^{T}(s)\rvert^{2}_{h(t)}-
    \lvert\dot{c}^{T}(s)\rvert^{2}_{h\bigl(c^{\mathrm{N}}(s)\bigr)}
  \end{align}
  The difference $\lvert\dot{c}^{T}(s)\rvert^{2}_{h(t)}-
  \lvert\dot{c}^{T}(s)\rvert^{2}_{h\bigl(c^{\mathrm{N}}(s)\bigr)}$ is bounded above by
  \begin{equation}
    \lvert h(t)-h\bigl(c^{\mathrm{N}}(s)\bigr)\rvert_{h\bigl(c^{\mathrm{N}}(s)\bigr)},
  \end{equation}
  which itself is bounded by $O(\lvert t-c^{\mathrm{N}}(s)\rvert)\leq O(t)$.
  From the definition of $\gamma^{t}_{p,q}$ we see that we have
  \begin{equation}
    \len_{g_{0}}(\gamma^{t}_{p,q})\leq d_{g_{0}}(p,q)+O(t)\cdot\mathrm{diam}_{g_{0}}(M).
  \end{equation}
\end{proof}

It was noted by Allen and Perales \cite{Allen-Perales} that when the boundary is convex with respect to $g_0$ the convex case is very similar to the compact case without boundary. Here we notice that by doing a conformal change near the boundary we can always make the boundary convex while preserving the metric on most of the manifold.

\begin{proposition}\label{prop:convexify_boundary}
  Let $(M,\partial{}M,g_{0})$ be a smooth Riemannian manifold with boundary,
  and let $t_{0}>0$ be any number such that $M_{t_{0}}\simeq\partial{}M\times{}[0,t_{0})$
  under the exponential map.
  Then, there is a smooth function $\phi:M\rightarrow{}[1,\infty{})$ such that
  $\partial{}M$ is convex with respect to $\widetilde{g}_{0}=\phi g_{0}$,
  and $\left.\phi\right|_{M_{t_{0}}^{c}}=1$.
\end{proposition}
\begin{proof}
  Let $A_{\widetilde{g}_{0}}$ denote the second fundamental of $\partial{}M$
  with respect to the metric $\widetilde{g}_{0}$, and let $\widetilde{\nu}$ denote
  the unit inward normal to $\partial{}M$ with respect to $\widetilde{g}_{0}$.
  Then, we have that
  \begin{equation}
    2A_{\widetilde{g}_{0}}=-\mathcal{L}_{\widetilde{\nu}}(\phi g_{0})=
    -\left(\frac{\partial{}\phi}{\partial{}\widetilde{\nu}}g_{0}+\phi\mathcal{L}_{\widetilde{\nu}}(g_{0}) \right)
  \end{equation}
  Denote by $\|\mathcal{L}_{\widetilde{\nu}}(g_{0})\|_{g_{0}}$ the sup-norm of
  $\mathcal{L}_{\widetilde{\nu}}(g_{0})$ on $\partial{}M$ with respect to $g_{0}$.
  If
  \begin{equation}
    \frac{\partial{}\phi}{\partial{}\widetilde{\nu}}
    \leq
    -2\phi\|\mathcal{L}_{\widetilde{\nu}}(g_{0})\|_{g_{0}},
  \end{equation}
  then we see that $A_{\widetilde{g}_{0}}$ is positive definite,
  and so $\partial{}M$ is convex with respect to $\widetilde{g}_{0}$.
  One may check that $\mathcal{L}_{\widetilde{\nu}}(g_{0})=-\frac{2}{\sqrt{\phi}}A_{g_{0}}$, where
  $A_{g_{0}}$ denotes the second fundamental form of $\partial{}M$ with respect to $g_{0}$,
  and so
  $\|\mathcal{L}_{\widetilde{\nu}}(g_{0})\|_{g_{0}}=\frac{2}{\sqrt{\phi}}\|-A_{g_{0}}\|_{g_{0}}$.
  Furthermore, letting $\nu$ denote the unit normal vector field to $\partial{}M$ with respect
  to $g_{0}$, we have that $\frac{\partial{}\phi}{\partial\widetilde{\nu}}=
  \frac{1}{\sqrt{\phi}}\frac{\partial{}\phi}{\partial{}\nu}$.
  It remains to show that such a function $\phi$ exists.

  We may choose a smooth function $\tau:[0,\infty{})\rightarrow{}[1,\infty{})$ which satisfies
  the following properties:
  \begin{align}
    &\left.\tau\right|_{[t_{0},\infty)}=1;
      \\
    &\frac{d\tau}{dt}\leq 0;
      \\
    &\left.\frac{d\tau}{dt}\right|_{0}
      \leq
      -4\tau(0)\|-A_{g_{0}}\|_{g_{0}}.
  \end{align}
  Now, let $\phi(x)=\tau\circ d_{g_{0}}(x,\partial{}M)$; we may compute the following:
  \begin{equation}
    \left.\frac{\partial{}\phi}{\partial{}\widetilde{\nu}}\right|_{\partial{}M}=
    \left.\frac{1}{\sqrt{\phi}}\frac{\partial{}\phi}{\partial\nu}\right|_{\partial{}M}=
    \frac{\tau'(0)}{\sqrt{\tau(0)}}\leq-4\sqrt{\tau(0)}\|-A_{g_{0}}\|_{g_{0}}=
    -2\left.\phi\right|_{\partial{}M}\|\mathcal{L}_{\widetilde{\nu}}g_{0}\|_{g_{0}}.
  \end{equation}
\end{proof}

Now by taking advantage of Lemma \ref{lem:rough_approximate_geodesic}, Proposition
\ref{prop:convexify_boundary}, and point-wise convergence from \cite{Allen-Perales}
in the convex case we can show point-wise almost everywhere convergence on $M_t \times M_t$.
\begin{theorem}\label{thm:a.e_convergence_on_distance_level_sets}
  Let $(M,\partial{}M,g_{0})$ be a Riemannian manifold with boundary, and let $g_{i}$
  be a sequence of metrics which VADB converge to $g_{0}$.
  Denote by $d_{j}$ the distance function of $g_{j}$ and $d_{0}$ the distance function of
  $g_{0}$.
  Then, for every $t>0$ we have that $d_{j}\xrightarrow{a.e.}d_{0}$ on $M^{c}_{t}\times{}M^{c}_{t}$.
\end{theorem}
\begin{proof}
  Suppose the statement is false. Since $t_{1}<t_{0}$ implies that
  $M_{t_{0}}^{c}\subset{}M_{t_{1}}^{c}$ it follows that for any $t>0$ there exists an
  $\varepsilon{}>0$ and a measurable subset $B\subset{}M_{t}\times{}M_{t}$
  such that $\lvert{}B\rvert>0$ and
  \begin{equation}
    \limsup_{j\rightarrow{}\infty{}}\left.(d_{j}-d_{0})\right|_{B}>\varepsilon{}.
  \end{equation}
  Let us fix a $t_{0}>0$ small enough so that $M_{t}\simeq\partial{}M\times{}[0,t)$, and
  let $\varepsilon{}>0$ and $B\subset{}M_{t_{0}}\times{}M_{t_{0}}$ be as above.

  Using Lemma \ref{lem:rough_approximate_geodesic} there exists a $t_{1}\leq t_{0}$ such that
  for any $p,q\in{}M_{t_{0}}$ we may find a curve $\gamma_{p,q}$ satisfying 
  \begin{equation}
    \len_{g_{0}}\left(\gamma_{p,q}\right)\leq d_{g_{0}}(p,q)+
    \frac{\varepsilon{}}{2}
  \end{equation}
  which lies in $M_{t_{1}}^{c}$.
  We shall modify the metrics $g_{i}$ and $g_{0}$ as follows.
  Let $\phi$ be as in Proposition \ref{prop:convexify_boundary} with
  \begin{equation}
    \left.\phi\right|_{M_{t_{1}}^{c}}=1,
    \end{equation}
  and set $\widetilde{g}_{0}=\phi g_{0}$.
  The metric $\widetilde{g}_{0}$ is a smooth Riemannian metric on $(M,\partial{}M)$ such that
  $\partial{}M$ is convex.
  Furthermore, by our length estimate on $\gamma_{p,q}\subset{}M_{t_{1}}^{c}$ and the fact that
  $\left.\phi\right|_{M_{t_{1}}^{c}}=1$, we see that
  \begin{equation}\label{eq:distance_estimate}
    d_{\widetilde{g}_{0}}(p,q)\leq d_{g_{0}}(p,q)+\frac{\varepsilon{}}{2}.
  \end{equation}
  This estimate is independent of the $p$ and $q$ in $M_{t_{0}}^{c}$.

  Consider now the sequence of metrics $\widetilde{g}_{i}=\phi g_{i}$,
  and observe that they VADB converge to $\widetilde{g}_{0}=\phi g_{0}$.
  Since $\partial{}M$ is convex with respect to $\widetilde{g}_{0}$, it follows from
  the results in \cite{Allen-Perales} Theorem \ref{thm-Pointwise ae Convergence Convex Boundary} that

    \begin{equation}
    d_{\widetilde{g}_{j}}(p,q)\rightarrow
    d_{\widetilde{g}_{0}}(p,q),  \qquad \textrm{ for } dV_{\widetilde{g}_{0}}\otimes{}dV_{\widetilde{g}_{0}} \text{ a.e. } (p,q) \in M \times M.
  \end{equation}
  The above almost everywhere convergence combined with \eqref{eq:distance_estimate} shows that
  \begin{equation}
    \limsup_{i\rightarrow{}\infty{}}d_{\widetilde{g}_{j}}\leq d_{g_{0}}+
    \frac{\varepsilon{}}{2}
  \end{equation}
  for $dV_{\widetilde{g}_{0}}\otimes{}dV_{\widetilde{g}_{0}}$ almost every
  $p$ and $q$ in $M_{t_0}^{c}$.
  Since $\phi$ is bounded above and below,
  we see that $dV_{g_{0}}$ and $dV_{\widetilde{g}_{0}}$ are
  equivalent measures, and so the above holds for $dV_{g_{0}}\otimes{}dV_{g_{0}}$ almost every
  $p$ and $q$ in $M_{t_{0}}^{c}$.
  However, as $\widetilde{g}_{j}\geq g_{j}$, this contradicts the existence of the set
  $B$, and so we have shown that
   \begin{equation}
    d_{g_{j}}(p,q)\rightarrow d_{g_{0}}(p,q)
  \end{equation}
 for  $dV_{g_{0}}\otimes{}dV_{g_{0}}$  a.e.  $(p,q) \in M_{t_{0}}\times{}M_{t_{0}}$.
\end{proof}

Now we observe that point-wise convergence on $M\times M$ follows quickly from Theorem \ref{thm:a.e_convergence_on_distance_level_sets}
\begin{corollary}
  Let $(M,\partial{}M,g_{0})$ be a Riemannian manifold with boundary, and let $g_{i}$ be 
  a sequence of metrics which VADB converge to $g_{0}$.
  Then $d_{g_{j}}$ converges to $d_{g_{0}}$ almost everywhere on $M\times{}M$ with respect to
  $dV_{g_{0}}\otimes{}dV_{g_{0}}$.
\end{corollary}
\begin{proof}
  Let $t_{n}$ be a sequence of positive numbers tending to zero, and for each $t_{n}$
  let $B_{n}\subset{}M_{t_{n}}^{c}\times{}M_{t_{n}}^{c}$ be the set on which
  $d_{g_{j}}$ does not converge to $d_{g_{0}}$.
  We see that $\lvert{}B_{n}{}\rvert_{}=0$ for all $n$ by Theorem
  \ref{thm:a.e_convergence_on_distance_level_sets}.
  Observe that $d_{g_{j}}$ converges on the set $\Bigl(\partial{}M\cup{}
  \bigcup_{n=1}^{\infty{}}B_{n}\Bigr)^{c}$,
  and $\partial{}M\cup{}\bigcup_{n=1}^{\infty{}}B_{n}$ has measure zero.
\end{proof}

\bibliographystyle{plain}
\bibliography{bibliography}

\end{document}